\documentclass[12pt]{amsart}
\usepackage{amsfonts,amsmath,amssymb,amscd,amsthm,enumerate,tikz}
\usepackage{amsrefs}
\usepackage{geometry}
\usepackage{adjustbox}
\usepackage[all]{xy}
\usepackage{stmaryrd}
\usetikzlibrary{arrows.meta,decorations.markings}

\newtheorem{theorem}{Theorem}
\newtheorem{prop}[theorem]{Proposition}
\newtheorem{lem}[theorem]{Lemma}
\newtheorem{cor}[theorem]{Corollary}
\theoremstyle{definition}

\newtheorem{rem}[theorem]{Remark}
\newtheorem{mydef}[theorem]{Definition}
\newtheorem{example}[theorem]{Example}

\renewcommand{\epsilon}{\varepsilon}

\def\<{\langle}
\def\>{\rangle}

\newcommand{\BigWedge}{\mathord{\adjustbox{valign=B,totalheight=.7\baselineskip}{$\bigwedge$}}}

\usepackage{multicol} 

\usepackage{tikz}
\usepackage{caption}
\usepackage{subcaption}

\begin{document}
\title[Independence Polynomials of 2-step Nilpotent Lie Algebras]{Independence Polynomials of 2-step Nilpotent Lie Algebras}
\author[M.\ Aldi, T.\ Gabrielsen, D.\ Grandini, J.\ Harris, K.\ Kelley]{Marco Aldi, Thor Gabrielsen, Daniele Grandini, Joy Harris, Kyle Kelley}

\begin{abstract}
Motivated by the Dani-Mainkar construction, we extend the notion of independence polynomial of graphs to arbitrary 2-step nilpotent Lie algebras. After establishing efficiently computable upper and lower bounds for the independence number, we discuss a metric-dependent generalization motivated by a quantum mechanical interpretation of our construction. As an application, we derive elementary bounds for the dimension of abelian subalgebras of 2-step nilpotent Lie algebras. 
\end{abstract}
\maketitle

\section{Introduction} 

The Dani-Mainkar construction \cite{DaniMainkar05} assigns to each finite simple graph $G$ a finite-dimensional 2-step nilpotent Lie algebra $\mathcal L(G)$. Compared to general 2-step nilpotent Lie algebras, Dani-Mainkar Lie algebras are special due to their combinatorial nature. For instance, the canonical basis labeled by vertices and edges of the corresponding graphs can be exploited to derive an explicit graph-theoretic description of the cohomology of Dani-Mainkar Lie algebras\cites{PouseeleTirao09,ABGGLP}.   

While this paper can be viewed as the natural continuation of the study of the Dani-Mainkar construction started in \cites{AldiBevins24, ABGGLP}, here we shift our focus from Dani-Mainkar Lie algebras to arbitrary finite-dimensional 2-step nilpotent Lie algebras. Our guiding principle is that Mainkar's theorem \cite{Mainkar15} justifies viewing 2-step nilpotent Lie algebras as generalizations of graphs. Specifically, we propose to regard 2-step nilpotent Lie algebras as ``quantum'' generalizations of graphs in which possibly non-trivial linear superpositions of classical edges are allowed. Guided by this philosophy, we aim to prove (or disprove)  generalizations of graph-theoretic results to arbitrary finite-dimensional 2-step nilpotent Lie algebras along the following lines. We start by recasting some graph-theoretic notions in purely Lie-theoretic language in a way that is compatible with the Dani-Mainkar construction. For instance, \cite{PouseeleTirao09} strongly suggests that the first Betti number and the dimension of the commutator ideal are the correct generalizations of the number of vertices and, respectively, of the number of edges. Once some entries of this dictionary between graph theory and Lie theory are established, it is then possible to take graph-theoretic results formulated using only notions admitting a Lie-theoretic counterpart and ask whether they are applicable to arbitrary 2-step nilpotent Lie algebras. 

As proof-of-concept, in this paper we show how this programme might be carried out in the context of the theory of independent sets of graphs. Independent sets, sets of vertices no two of which are joined by an edge, have been extensively studied in graph theory. Our first observation is that the independent sets of a simple graph $G$ provide a canonical basis for the sector of the cohomology of $\mathcal L(G)$ that is pulled-back via the canonical projection map onto its abelianization. Since this sector of the cohomology, which for geometric reasons we refer to as the basic cohomology, makes sense for arbitrary 2-step nilpotent Lie algebras, we take it as our starting point for the proposed generalization of the independence theory of graphs beyond the Dani-Mainkar setting. In particular, using the dimension of the graded pieces of the basic cohomology as coefficients, we attach a single-variable polynomial to every 2-step nilpotent Lie algebra in such a way as to recover the independence polynomial of a graph in the Dani-Mainkar case. Accordingly, we refer to this polynomial as the independence polynomial of a 2-step nilpotent Lie algebra, and explicitly calculate it in the (non-Dani-Mainkar) case of Heisenberg Lie algebras. 

The problem of deciding if an independent set of given size exists (or, dually, if a clique of given size exists) is known to be NP-complete \cite{Karp72}. To narrow down the search space, it is helpful to have efficiently computable upper and lower bounds, for the independence number i.e.\ the degree of the independence polynomial. In this paper we focus on an upper bound due to Hansen \cite{Hansen79} and on a lower bound which is implied by a theorem of Turan \cite{Turan41}. Both of these bounds are algebraic functions of the number of edges and the number of vertices. In particular, both bounds on the independence number of a graph are efficiently computable and it makes sense to ask whether they hold beyond the Dani-Mainkar case. We show that, suitably restated in terms of the first Betti number and the dimension of the Lie algebra, the Hansen upper bound is indeed valid for arbitrary 2-step nilpotent Lie algebras. As an application, we derive an efficiently computable upper bound for the dimension of abelian subalgebras of 2-step nilpotent Lie algebras. 

On the other hand, as the example of Heisenberg Lie algebras shows, the lower bound on the independence number of a graph coming from Turan's theorem fails for general 2-step nilpotent Lie algebras. Instead, we are able to give a different lower bound, also an algebraic function of the first Betti number and the dimension of the Lie algebra, for the independence number of arbitrary finite-dimensional 2-step nilpotent Lie algebras. As shown in a companion paper \cite{AGGHK}, in the case of graphs this new lower bound is always dominated by the Turan lower bound. However, its natural (in light of \cite{AldiBevins24}) extension to $L_\infty$-algebras restricts to a novel lower bound on the independence number of hypergraphs. The lower bound found in \cite{AGGHK} is a combinatorial result first obtained by Lie-theoretical reasoning. Together with the Lie-theoretic results obtained by graph-theoretic reasoning in this paper, this is a good illustration of the potential impact that this line of inquiry can have in these two seemingly unrelated areas of mathematics.  

We conclude our paper with a quantum mechanical interpretation of the basic cohomology of a 2-step nilpotent Lie algebra. To this end, we introduce an inner product on the vector space underlying the Lie algebra and use it to define a Laplacian operator acting on corresponding Cartan-Chevalley-Eilenberg complex. Equivalently, we realize (non-canonically, due to the choice of the inner product) the Cartan-Chevalley-Eilenberg complex as the Hilbert space of a supersymmetric quantum mechanical system with purely fermionic degrees of freedom. In this picture, we are able to identify the basic cohomology as the ground states of a specific sector of the Hilbert space (in the Dani-Mainkar case, these are the fermionic states labeled by vertices of the underlying graph). 

The information contained in the spectrum of this sector is naturally encoded by the so-called basic partition function attached to the given 2-step nilpotent Lie algebra and the chosen inner product. Our first observation is that the independence polynomial is a univariate specialization of this bivariate basic partition function. Furthermore, in the Dani-Mainkar case the basic partition function is a specialization of the four-variable generalized subgraph counting polynomial introduced in \cite{Trinks12}. We also calculate the basic partition function for arbitrary Heisenberg Lie algebras with respect to the inner product that makes the standard basis orthonormal. Intriguingly, the resulting explicit formula uses in an essential way the spectrum of Johnson graphs. 

We believe that the results presented in this paper provide sufficient evidence to justify further use of graph theory as a guiding metaphor in the study of 2-step nilpotent Lie algebras and their related supersymmetric quantum mechanical systems. We leave further progress in the programme sketched here to future work.

\section{Preliminaries}

\subsection{Independence Polynomials of Graphs}

In this section we collect known facts about the theory of independent sets of graphs and the Dani-Mainkar construction relating graphs to 2-step nilpotent Lie algebras.

\begin{mydef}
Let $G$ be a finite simple graph with vertices $V(G)$. An {\it independent set} of $G$ is a subset $S\subseteq V(G)$ whose induced subgraph $G[S]$ contains no edges i.e.\ it is isomorphic to $|S|K_1$. We denote by $s_k(G)$ the number of independent sets of size $k$ of $G$. The {\it independence number of $G$} is the maximum, denoted $\alpha(G)$, of the set of all integers $k$ such that $s_k\neq 0$. The {\it independence polynomial of $G$} is the polynomial
\begin{equation}
I(G,t)=\sum_{k=0}^{\alpha(G)} s_k(G) t^k\,.
\end{equation}
\end{mydef}

\begin{example}[\cite{Arocha84}]
If $P_n$ is the path graph with $n$ vertices, then
\begin{equation}
I(P_n,t)=\sum_{k=0}^{\lfloor(n+1)/2\rfloor} \binom{n+1-k}{k} t^k\,.
\end{equation}
\end{example}

\begin{theorem}[\cites{Berge76,Hansen79}]
Let $G$ be a finite simple graph with vertices $V(G)$ and edges $E(G)$. Then
\begin{equation}\label{eq:3}
\frac{|V(G)|^2}{2|E(G)|+|V(G)|}\le \alpha(G)\le \frac{1}{2}+\sqrt{ \frac{1}{4}+|V(G)|^2-|V(G)|-2|E(G)|}\,.
\end{equation}
\end{theorem}

\begin{mydef}[\cite{Trinks12}]\label{def:4}
The {\it generalized subgraph counting polynomial} of a finite simple graph $G$ is
\begin{equation}\label{eq:4}
F(G,q,r,s,t)=\sum_{H\subseteq G} q^{k(H)}r^{|E(H)|}s^{|E(G[V(H)])|}t^{|V(H)|}\,,
\end{equation}
where the sum is over all (not necessarily induced) subgraphs $H$ of $G$, $k(H)$ is the number of connected components of $H$ and $G[V(H)]$ and is the subgraph of $G$ induced by the vertices of $H$.
\end{mydef}

\begin{rem}
$F(G,1,1,0,t)=I(G,t)$ for every finite simple graph $G$.
\end{rem}

\subsection{2-step nilpotent Lie algebras}

\begin{mydef}
A (real) {\it 2-step nilpotent Lie algebra} is a vector space $\mathfrak g$ over $\mathbb R$ together with an antisymmetric bilinear operation, known as the {\it Lie bracket}, $[\,,\,]:\mathfrak g\times \mathfrak g\to \mathfrak g$ such that $[x,[y,z]]=0$ for every $x,y,z\in \mathfrak g$.
\end{mydef}

\begin{example}\label{ex:5}
The {\it $(2n+1)$-dimensional Heisenberg Lie algebra} is the 2-step nilpotent Lie algebra $\mathfrak h_{n}$ with basis $y_1,\ldots,y_{2n},z$ and Lie bracket such that $[y_{2k-1},y_{2k}]=z$ whenever $k\in\{1,\ldots,n\}$. 
\end{example}

\begin{mydef}
Let $\mathfrak g$ be a 2-step nilpotent Lie algebra. The {\it Cartan-Chevalley-Eilenberg complex of $\mathfrak g$} is the exterior algebra $\mathcal C^\bullet(\mathfrak g)$ of the dual $\mathfrak g^\vee$ of $\mathfrak g$ together with the odd derivation $d$ such that $d(\varphi)(x,y)=\varphi([x,y])$ for every $\varphi\in \mathfrak g^\vee$ and $x,y\in \mathfrak g$. The {\it cohomology of $\mathfrak g$} is the cohomology $H^\bullet(\mathfrak g)$ of the cochain complex $(\mathcal C^\bullet(\mathfrak g),d)$. The $k$-th {\it Betti number of $\mathfrak g$} is $b^k(\mathfrak g)=\dim H^k(\mathfrak g)$.
\end{mydef}

\begin{example}
The first Betti number $b^1(\mathfrak g)$ is equal to the dimension of the space of elements $\varphi\in \mathfrak g^\vee$ such that $[\mathfrak g,\mathfrak g]\subseteq \ker(\varphi)$.
\end{example}

\begin{example}[\cite{Santharoubane83}] Let $\mathfrak h_n$ be as in Example \ref{ex:5}. Then, for every $k\in \{0,\ldots, n\}$,
\begin{equation}
b^{2n+1-k}(\mathfrak g)=b^k(\mathfrak g)=\binom{2n}{k}-\binom{2n}{k-2}\,.
\end{equation}
\end{example}

\begin{rem}\label{rem:10}
Let $\mathfrak g$ be a 2-step nilpotent Lie algebra with basis $\{y_1,\ldots,y_b,z_1,\ldots,z_c\}$ and brackets $[y_i,y_j]=\sum_{k=1}^c \gamma_{i,j}^k z_k$ for some real structure constants $\gamma_{i,j}^k$. If the dual basis is $\{y_1^*,\ldots,y_b^*,z_1^*,\ldots,z_c^*\}$, then the Cartan-Chevalley-Eilenberg complex $\mathcal C^\bullet(\mathfrak g)$ can be concretely realized as the algebra of anticommutative polynomials in the variables $y_1^*,\ldots,y_b^*,z_1^*,\ldots,z_c^*$. With respect to these variables, the differential of $\mathcal C^\bullet(\mathfrak g)$ can be written explicitly as a first-order differential operator
\begin{equation}\label{eq:6}
d=\sum_{k=1}^c \sum_{1\le i<j\le b} \gamma_{i,j}^k y_i^*y_j^* \frac{\partial}{\partial z_k^*}\,.
\end{equation}
\end{rem}

\subsection{The Dani-Mainkar Construction}

\begin{mydef}[\cite{DaniMainkar05}]
Let $G$ be a finite simple graph with vertices $V(G)=\{1,\ldots,n\}$ and edges $E(G)$. Let $V$ be the (real) vector space with basis $\{x_1,\ldots,x_n\}$ and let $W$ be the subspace of $\BigWedge^2 V$ generated by monomials $x_i\wedge x_j$ whenever $\{i,j\}$ is not in $E(G)$. The {\it Dani-Mainkar Lie algebra of $G$} is the 2-step nilpotent Lie algebra $\mathcal L(G)=V\oplus \left(\BigWedge^2 V \right)/W$ with Lie bracket such that $[x,y]=x\wedge y \mod W$ for all $x,y\in V$. 
\end{mydef}

\begin{example}
If $K_2$ is the complete graph on $2$ vertices, then $\mathcal L(K_2)$ is isomorphic to the $3$-dimensional Heisenberg Lie algebra $\mathfrak h_1$.
\end{example}

\begin{theorem}[\cite{Mainkar15}]\label{thm:10}
Let $G_1$ and $G_2$ be finite simple graphs. Then $G_1$ and $G_2$ are isomorphic if and only if $\mathcal L(G_1)$ and $\mathcal L(G_2)$ are isomorphic as Lie algebras. 
\end{theorem}

\begin{rem}
As shown in \cite{ABGGLP}, the Betti numbers of a Dani-Mainkar Lie algebra $\mathcal L(G)$ can be expressed as a weighted count of isomorphism classes of graphs occurring as induced subgraphs of $G$. In particular, independent sets of size $k$ of $G$ contribute (with weight $1$) to $b^k(G)$. 
\end{rem}

\begin{rem}\label{rem:12}
As observed in \cite{PouseeleTirao09}, $|V(G)|=b^1(\mathcal L(G))$ and
\begin{equation}
    |E(G)|=\dim([\mathcal L(G),\mathcal L(G)])=\dim(\mathcal L(G))-b^1(\mathcal L(G))
\end{equation} for every finite simple graph $G$. This suggests viewing $b^1(\mathfrak g)$ and $\dim(\mathfrak g)-b^1(\mathfrak g)$ as generalizations of, respectively, the notion of the number of vertices and the number of edges for an arbitrary 2-step nilpotent Lie algebra $\mathfrak g$.
\end{rem}

\begin{rem}
The Dani-Mainkar Construction has been extended to hypergraphs in \cite{AldiBevins24} by letting $\mathcal L(G)$ be an (analogously defined) 2-step nilpotent $L_\infty$-algebra for every simple finite hypergraph $G$. Theorem \ref{thm:10} generalizes to this setting. 
\end{rem}

\section{Independence Polynomials of 2-step nilpotent Lie algebras}

In this section we introduce the basic cohomology of an arbitrary 2-step nilpotent Lie algebra and use it to define the independence polynomial. The material of this section extend straightforwardly to 2-step nilpotent $L_\infty$-algebras.

\begin{mydef}
Let $\mathfrak g$ be a 2-step nilpotent Lie algebra and let $\pi:\mathfrak g\to \mathfrak g/[\mathfrak g,\mathfrak g]$ be the canonical quotient map. The {\it basic cohomology of $\mathfrak g$} is the image $H_B^\bullet(\mathfrak g)$ of the induced linear map $\pi^\bullet:H^\bullet(\mathfrak g/[\mathfrak g,\mathfrak g])\to H^\bullet(\mathfrak g)$. We define the {\it $k$-th basic Betti number of $\mathfrak g$} to be $b_B^k(\mathfrak g)=\dim H_B^k(\mathfrak g)$. We use the notation $\mathcal C^\bullet_B(\mathfrak g)$ for the subalgebra $\BigWedge^\bullet H^1(\mathfrak g)$ of $\mathcal C^\bullet(\mathfrak g)$.
\end{mydef}

\begin{rem}
The terminology is justified by the topological interpretation  of the cohomology of 2-step nilpotent Lie algebras as the cohomology of the associated compact nilmanifolds due to Nomizu \cite{Nomizu54}. Since compact nilmanifolds associated with a 2-step nilpotent Lie algebras can be realized as torus fibrations over a torus, the basic cohomology can be thought of as consisting of cohomology classes pulled back from the base of the fibration by means of the natural projection map.
\end{rem}

\begin{rem}\label{rem:15}
Let $\mathfrak g$ be a 2-step nilpotent Lie algebra with basis $\{y_1,\ldots,y_b,z_1,\ldots,z_c\}$ obtained by extending a basis $\{z_1,\ldots,z_c\}$ of the ideal $[\mathfrak g,\mathfrak g]$ so that, in particular, $b=b^1(\mathfrak g)$. Viewing the elements of $\mathcal C^\bullet(\mathfrak g)$ as polynomials in the dual variables in accordance with Remark \ref{rem:10}, we can naturally identify $\mathcal C^\bullet_B(\mathfrak g)$ with polynomials in the $y_i^*$ variables only. Moreover, $H^k_B(\mathfrak g)$ is isomorphic to the quotient of $\mathcal C^\bullet_B(\mathfrak g)$ by the subspace spanned by polynomials of the form $d(z_j^*y_{i_1}^*\cdots y_{i_{k-2}}^*)$.
\end{rem}

\begin{prop}
Let $G$ be a finite simple graph and let $\mathcal L(G)$ be its Dani-Mainkar Lie algebra. Then $b^k_B(\mathcal L(G))=s_k(G)$ for all $k\ge 0$.
\end{prop}

\begin{proof}
If we label the vertices of $G$ by setting $V(G)=\{1,\ldots,n\}$, then $\mathcal L(G)$ has a basis consisting of $x_i$ for each $i\in V(G)$ and $x_{ij}$ for each $i<j$ such that $\{i,j\}\in E(G)$. By Remark \ref{rem:15}, $H_B^k(\mathcal L(G))$ is then isomorphic to the space of degree $k$ polynomials in the anticommuting variables $x_1^*,\ldots,x_n^*$ quotiented by the space of polynomials of the form
\begin{equation}
d(x_{ij}^*x_{i_1}^*\cdots x_{i_{k-2}}^*)=x_i^*x_j^*x_{i_1}^*\cdots x_{i_{k-2}}^*
\end{equation}
and thus to the space spanned by monomials of the form $x_{i_1}^*\cdots x_{i_k}^*$ such that $\{i_1,\ldots,i_k\}$ is an independent set of $G$.
\end{proof}

\begin{mydef}
The {\it independence number of a finite-dimensional 2-step nilpotent Lie algebra $\mathfrak g$} is the largest integer $\alpha(\mathfrak g)$ such that $b^{\alpha(\mathfrak g)}_B(\mathfrak g)\neq 0$. The {\it independence polynomial of a 2-step nilpotent Lie algebra $\mathfrak g$} is
\begin{equation}
I(\mathfrak g,t)=\sum_{k=0}^{\alpha(\mathfrak g)} b^k_B(\mathfrak g) t^k\,.
\end{equation}
\end{mydef}

\begin{mydef}
Let $\mathfrak g$ be a 2-step nilpotent Lie algebra. An {\it independent set of $\mathfrak g$} is a subset $S=\{\varphi_1,\ldots,\varphi_k\}\subseteq H^1(\mathfrak g)$ such that $\det(S)=\varphi_1\wedge\cdots\wedge\varphi_k$ is nonzero in $H^{k}_B(\mathfrak g)$.
\end{mydef}

\begin{prop}\label{prop:20} Let $\mathfrak g$ be a 2-step nilpotent Lie algebra of dimension $d<\infty$ and first Betti number $b$. Then
\begin{enumerate}[1)]
\item The size of the largest independent set of $\mathfrak g$ is equal to $\alpha(\mathfrak g)$.
\item If $\mathfrak h$ is an abelian subalgebra of $\mathfrak g$, then $\dim(\mathfrak h)\le\alpha(\mathfrak g)+d-b$.
\end{enumerate}
\end{prop}

\begin{proof}
If $S$ is an independent set of a 2-step nilpotent Lie algebra $\mathfrak g$, then, by definition, $\det(S)$ is a non-zero element of $H^{|S|}_B(\mathfrak g)$ and thus $|S|\le \alpha(\mathfrak g)$. For the reverse inequality, fix a basis $\{y_1,\ldots,y_b,z_1,\ldots,z_c\}$ of $\mathfrak g$ as in Remark \ref{rem:15}. Let $\omega$ be a polynomial representing a non-zero class in $H_B^{\alpha(\mathfrak g)}(\mathfrak g)$ and assume that among all other representatives of the same cohomology class $\omega$ has the least number of terms (i.e.\ it is a linear combination of the least possible number of monomials). Up to overall scaling, we may assume that one of these terms is $\omega'=y_{i_1}^*\cdots y_{i_{\alpha(\mathfrak g)}}^*$. If $S=\{y_{i_1}^*,\ldots,y_{i_{\alpha(\mathfrak g)}}^*\}$, then $\det(S)$ is a non-zero element of $H_B^{\alpha(\mathfrak g)}(\mathfrak g)$, for otherwise $\omega-\omega'$ would be a polynomial in the same cohomology class as $\omega$ but with fewer terms. Hence $\mathfrak g$ admits an independent set of dimension $\alpha(\mathfrak g)$, concluding the proof of 1). 

To prove 2), let $\mathfrak h$ be an abelian subalgebra of $\mathfrak g$. Let $\{y_1,\ldots,y_k\}$ be a basis of $\pi(\mathfrak h)$ and let $\{y_1,\ldots,y_b,z_1,\ldots,z_c\}$ be an extension of this basis to $\mathfrak g$. Let $\{y_1^*,\ldots,y_b^*,z_1^*,\ldots,z_c^*\}$ be the dual basis. We claim that $S=\{y_1^*,\ldots,y_k^*\}$ is an independent set of $\mathfrak g$. Suppose not i.e.\ $\det(S)=\sum_{i=1}^c(dz_i^*\omega_i)$ for some polynomials $\omega_i$ in the $y_j^*$ variables. Then on the one hand, when viewed as a $k$-form on $\mathfrak g$, $\det(S)$ take non-zero value only when evaluated on linearly independent vectors in $\pi(\mathfrak h)$. On the other hand, $\sum_{i=1}^c(dz_i^*\omega_i)$ necessarily vanishes on $\pi(\mathfrak h)$ since $\mathfrak h$ is abelian. This contradiction shows that $S$ is independent. Therefore, using 1), we obtain $\dim(\pi(\mathfrak h))=|S|\le \alpha(\mathfrak g)$ and thus
\begin{equation}
\dim(\mathfrak h)\le \dim(\pi(\mathfrak h))+\dim([\mathfrak g,\mathfrak g]) \le \alpha(\mathfrak g)+d-b\,.
\end{equation}

\end{proof}

\begin{example}
Consider the 7-dimensional 2-step nilpotent Lie algebra $\mathfrak g$ generated by $y_1,y_2,y_3,y_4,z_1,z_2,z_3$ with non-zero Lie brackets $[y_1,y_2]=z_1=[y_3,y_4]$, $[y_1,y_3]=z_2=[y_2,y_4]$, and $[y_1,y_4]=z_3=[y_2,y_3]$. One the one hand, $\mathfrak g$ has independence number equal to 2 and first Betti number equal to 4 so that $\alpha(\mathfrak g)+d-b=5$. On the other hand, maximal abelian subalgebras are of dimension 4 (generated by $z_1,z_2,z_3$ and a single linear combination of the form $a_1y_1+a_2y_2+a_3y_3+a_4y_4$). Hence in this example, the inequality in statement 2) of Proposition \eqref{prop:20} is strict. 
\end{example}

\begin{theorem}
The independence polynomial of the Heisenberg Lie algebra of dimension $2n+1$ is
\begin{equation}
I(\mathfrak h_n,t)= \sum_{k=0}^n \left(\binom{2n}{k}-\binom{2n}{k-2}\right) t^k\,.
\end{equation}
\end{theorem}

\begin{proof}
Consider first the case $k\le n$. Let $S_k$ be the $\binom{2n}{k}$ space of degree $k$ polynomials in the anticommuting variables $y_1^*,\ldots,y_{2n}^*$. By Remark \ref{rem:15}, $H_B^k(\mathfrak h_n)$ is isomorphic to the quotient of $S_k$ by its subspace $d(zS_k)$. Hence, $b_B^k(\mathfrak h_n)\ge \binom{2n}{k}-\binom{2n}{k-2}=b^k(\mathfrak h_n)$. On the other hand, by definition of basic cohomology,  $b_B^k(\mathfrak h_n)\le b^k(\mathfrak h_n)$ for all $k$. Therefore, $b_B^k(\mathfrak h_n)= b^k(\mathfrak h_n)$ for all $k\le n$. Equivalently, all cohomology classes in degree less or equal than $n$ are basic. By Poincar\'e duality, this means that none of the cohomology classes in degree greater than $k$ are i.e.\ $b_B^k(\mathfrak h_n)=0$ for all $k>n$.
\end{proof}

\begin{cor}\label{cor:20}
$\alpha(\mathfrak h_n)=n$.
\end{cor}

\section{Bounds on the Independence Number}

In this section we establish efficiently computable upper and lower bounds for the independence number of an arbitrary finite-dimensional 2-step nilpotent Lie algebra.

\begin{theorem}\label{thm:27}
Let $\mathfrak g$ be 2-step nilpotent Lie algebra of dimension $d<\infty$ whose first Betti number is equal to $b$. Then
\begin{equation}\label{eq:9}
\alpha(\mathfrak g) \le \frac{1}{2}+\sqrt{\frac{1}{4}+b^2+b-2d}\,.
\end{equation}
\end{theorem}

\begin{proof} By Proposition \ref{prop:20}, there exists an independent set $S$ of $\mathfrak g$ such that  $|S|=\alpha(\mathfrak g)$. Since $S$ has non-zero determinant, $\BigWedge^2 {\rm span}(S)$ is subspace of  $H_B^2(\mathfrak g)$ of dimension $\binom{\alpha(\mathfrak g)}{2}$. Therefore, by Remark \ref{rem:15}, 
\begin{equation}
(\alpha(\mathfrak g))^2-\alpha(\mathfrak g)\le 2b_B^2(\mathfrak g)\le 2\binom{b}{2}-2\dim([\mathfrak g,\mathfrak g])=b^2+b-2d
\end{equation}
from which \eqref{eq:9} easily follows.
\end{proof}

\begin{cor}
Let $\mathfrak g$ be 2-step nilpotent Lie algebra of dimension $d<\infty$ whose first Betti number is equal to $b$. If $\mathfrak h$ is an abelian subalgebra of $\mathfrak g$, then 
\begin{equation}
\dim(\mathfrak h)\le \frac{1}{2}+\sqrt{\frac{1}{4}+b^2+b-2d}+d-b\,.
\end{equation}
\end{cor}

\begin{proof}
Combine Theorem \ref{thm:27} with Proposition \ref{prop:20}.
\end{proof}

\begin{rem}\label{rem:22}
It follows from Remark \ref{rem:12} that in the Dani-Mainkar case \eqref{eq:9} specializes to the upper bound in \eqref{eq:3}.
\end{rem}

\begin{rem}
In light of Remark \ref{rem:22}, it is natural to ask whether the lower bound in \eqref{eq:3} also generalizes to 2-step nilpotent Lie algebras i.e.\ if
\begin{equation}\label{eq:12}
\frac{b^2}{2d-b}\le \alpha(\mathfrak g)
\end{equation}
whenever $\mathfrak g$ is a 2-step Lie algebra of dimension $d<\infty$ whose first Betti number is equal to $b$. By Remark \ref{cor:20}, this is false for all $\mathfrak g=\mathfrak h_n$ with $n\ge 2$ since in this case the RHS of \eqref{eq:12} is equal to $\frac{2n^2}{n+1}>n$. Instead, we have the following:
\end{rem}

\begin{theorem}\label{thm:24}
    Let $\mathfrak g$ be a 2-step nilpotent Lie algebra of dimension $d<\infty$ and  with first Betti equal to $b$. 
\begin{enumerate}[1)]    
\item If $d<b+1$, then $\alpha(\mathfrak g)=d$.
\item If $d=b+1$, then  $\alpha(\mathfrak g)\ge \frac{d-1}{2}$.
\item If $d>b+1$, then  
    \begin{equation}\label{eq:13}
    \alpha(\mathfrak g)\ge \frac{\sqrt{4(d-b-1)(b^2+b)+(d+b+1)^2}-(d+b+1)}{2(d-b-1)}\,.
\end{equation}
\end{enumerate}
\end{theorem}

\begin{proof} If $d<b+1$, then $\mathfrak g$ is abelian and thus $\alpha(\mathfrak g)=d$. Assume $d\ge b+1$. By definition of independence number, $H_B^{\alpha(\mathfrak g)+1}(\mathfrak g)=0$. Given a basis $\{y_1,\ldots,y_b,z_1,\ldots,z_c\}$ of $\mathfrak g$ as in Remark \ref{rem:15}, we conclude that every polynomial of degree $\alpha(\mathfrak g)+1$ in the anticommuting variables $y_1^*,\ldots,y_b^*$ is of the form $\sum_{i=1}^c (dz_i^*) \omega_i$ where each $\omega_i$ is a degree $\alpha(\mathfrak g)-1$ polynomial in the variables $y_1^*,\ldots,y_b^*$. Since the dimension of the domain of every surjective linear transformation is at least the dimension of its corresponding codomain, we obtain
\begin{equation}
\binom{b}{\alpha(\mathfrak g)+1}\le (d-b) \binom{b}{\alpha(\mathfrak g)-1}
\end{equation}
which, combined with straightforward algebraic manipulations, proves 2) and 3).
\end{proof}

\begin{example}
If $\mathfrak g=\mathfrak h_n$, then $d=b+1$ and Corollary \ref{cor:20} implies that lower bound established by Theorem \ref{thm:24} is exact.
\end{example}

\begin{rem}
When $\mathfrak g$ is the Dani-Mainkar Lie algebra of a finite simple graph $G$, \eqref{eq:13} gives an efficiently computable lower bound for the independence number of graphs. As shown in \cite{AGGHK}, this lower bound is always dominated by the lower bound in \eqref{eq:3}. Nevertheless, the proof of Theorem \ref{thm:24} easily extends to $k$-uniform (i.e.\ with only $k$-ary operations) 2-step nilpotent $L_\infty$-algebras. For some $k$-uniform hypergraphs, the resulting lower bound (for which an alternate combinatorial proof was supplied in \cite{AGGHK}) improves on known results in the literature on independence number of hypergraphs \cites{CaroTuza91, CsabaPlickShokoufandeh12, Eustis13}.
\end{rem}

\section{The basic Laplacian}

In this section we introduce a Laplacian operator which depends on the additional datum of an inner product and study the partition function of the associated supersymmetric quantum mechanics. Most of the material of this section can be straightforwardly generalized to finite-dimensional 2-step nilpotent $L_\infty$-algebras.

\begin{mydef} A {\it metric pair} $(\mathfrak g,g)$ consists of a finite-dimensional 2-step nilpotent Lie algebra $\mathfrak g$ and a (positive definite) inner product $g$ on the real vector space underlying $\mathfrak g$ (we do not require compatibility with the Lie bracket). Then $g$ defines a Hodge star operator $\star$ on $\mathcal C^\bullet(\mathfrak g)$ and consequently the {\it adjoint Cartan-Chevalley-Eilenberg operator} $d^*=\star^{-1}d\star$. The corresponding {\it Laplacian operator} acting on $\mathcal C^\bullet(\mathfrak g)$ is $\Delta=dd^*+d^*d$. We define the {\it basic Laplacian operator} as the restriction $\Delta_B$ of $\Delta$ to $\mathcal C^\bullet_B(\mathfrak g)$.
\end{mydef}

\begin{rem}\label{rem:33} 
Let $(\mathfrak g,g)$ be a metric pair and let orthonormal basis $\{y_1,\ldots,y_b,z_1,\ldots,z_c\}$ of $\mathfrak g$ with respect to which the Lie bracket can be written as $[y_i,y_j]=\sum_{k=1}^c \gamma_{i,j}^k z_k$ for some real structure constants $\gamma_{i,j}^k$. Then $d^*$ acts on the dual variables as the second-order differential operator
\begin{equation}\label{eq:16}
d^*= \sum_{k=1}^c \sum_{1\le i<j\le b} \gamma_{i,j}^k z_k^* \frac{\partial}{\partial y_j^*} \frac{\partial}{\partial y_i^*}\,.
\end{equation}
Taking into account that $d$ acts trivially on $\mathcal C_B^\bullet(\mathfrak g)$, combining \eqref{eq:6} with \eqref{eq:16} we obtain
\begin{equation}\label{eq:17}
\Delta_B=dd^*=\sum_{k=1}^c \sum_{1\le i<j\le b} \sum_{1\le r<s\le b} \gamma_{i,j}^k \gamma_{r,s}^k y_r^* y_s^* \frac{\partial}{\partial y_j^*} \frac{\partial}{\partial y_i^*}\,. 
\end{equation}
\end{rem}

\begin{example}
If $G$ is a finite simple graph, then $\mathcal L(G)$ comes equipped with the canonical inner product $g_G$ with respect to which the basis labeled by vertices and edges of $G$ is orthonormal. Then \eqref{eq:17} specializes to 
\begin{equation}\label{eq:18}
\Delta_B=\sum x_i^* x_j^* \frac{\partial}{\partial x_j^*}\frac{\partial}{\partial x_i^*}\,,
\end{equation}
where the sum is extended over all $i,j\in \{1,\ldots, |V(G)|\}$ such that $i<j$ and $\{i,j\}\in E(G)$.
\end{example}

\begin{example}\label{ex:35}
Consider the metric pair $(\mathfrak h_2,h)$ with $h$ the inner product with respect to which the basis $\{y_1,y_2,y_3,y_4,z\}$ of Example \ref{ex:5} is orthonormal. The basic Laplacian of $(\mathfrak h_2,h)$
\begin{equation}
\Delta_B=(y_1^*y_2^*+y_3^*y_4^*)\left(  \frac{\partial}{\partial y_2^*} \frac{\partial}{\partial y_1^*}+ \frac{\partial}{\partial y_4^*} \frac{\partial}{\partial y_3^*}\right)\,.
\end{equation}
Let $\mathfrak g$ be the 2-step nilpotent Lie algebra with basis $\{u_1,u_2,u_3,u_4,w\}$ and non-zero brackets $[u_1,u_2]=[u_2,u_3]=[u_3,u_4]=w$ and let $g$ be the inner product with respect to which $\{u_1,u_2,u_3,u_4,w\}$ is orthonormal. The change of variables $y_1=u_1$, $y_2=u_2+\frac{1}{2}u_4$, $y_3=u_3+\frac{1}{2}u_1$, $y_4=u_4$, and $z=w$ shows that $\mathfrak g$ is isomorphic to $\mathfrak h_2$. Through this identification, the inner product $g$ can be equivalently thought of as the inner product $h'$ on $\mathfrak h_2$ represented by the matrix
\begin{equation}
\begin{bmatrix}
1 & 0 & \frac{1}{2} & 0 & 0\\
0 & \frac{5}{4} & 0 & \frac{1}{2} & 0\\
\frac{1}{2} & 0 & \frac{5}{4} & 0 & 0\\
0 & \frac{1}{2} & 0 & 1 & 0\\
0 & 0 & 0 &0 &1
\end{bmatrix}
\end{equation}
with respect to the basis $\{y_1,y_2,y_3,y_4,z\}$. As an illustration of the dependence of the basic Laplacian on the choice of inner product, for the metric pair $(\mathfrak g,g)=(\mathfrak h_2,h')$ we obtain
\begin{align*}
\Delta_B&=(u_1^*u_2^*+u_2^*u_3^*+u_3^*u_4^*)\left(\frac{\partial}{\partial u_2^*} \frac{\partial}{\partial u_1^*}+\frac{\partial}{\partial u_3^*} \frac{\partial}{\partial u_2^*}+\frac{\partial}{\partial u_4^*} \frac{\partial}{\partial u_3^*} \right)\\
&=(y_1^*y_2^*+y_3^*y_4^*)\left( \frac{3}{2} \frac{\partial}{\partial y_2^*} \frac{\partial}{\partial y_1^*}+\frac{3}{2} \frac{\partial}{\partial y_4^*} \frac{\partial}{\partial y_3^*}+ \frac{\partial}{\partial y_3^*} \frac{\partial}{\partial y_2^*}- \frac{5}{4}\frac{\partial}{\partial y_4^*} \frac{\partial}{\partial y_1^*} \right)\,.
\end{align*}
\end{example}

\begin{mydef}
Let $(\mathfrak g,g)$ be a metric pair and let $\Delta_B$ be the corresponding basic Laplacian operator. The {\it basic partition function of $(\mathfrak g,g)$} is the trace
\begin{equation}\label{eq:21}
Z_{\mathfrak g,g}(s,t)={\rm Tr}(s^{\Delta_B}t^{\deg})
\end{equation}
taken over $\mathcal C^\bullet_B(\mathfrak g)$, where the exponent of $t$ denotes the {\it degree} operator diagonalized by setting $\deg\omega=m\omega$ whenever $\omega$ is a monomial of degree $m$. 
\end{mydef}

\begin{rem}
In the notation of Remark \ref{rem:33}, $\mathcal C^\bullet(\mathfrak g)$ can be interpreted as the Hilbert space of states of a quantum mechanical system consisting of two species of fermions (namely, $y_1^*,\ldots,y_b^*$ and $z_1^*,\ldots,z_c^*$) and supersymmetric Hamiltonian $H=\frac{1}{2}\Delta$. A standard argument (see e.g.\ \cite{Hori03}) shows that $H^\bullet(\mathfrak g)$ is isomorphic to $\ker(\Delta)$, offering an alternate approach to the calculation of the cohomology of 2-step nilpotent Lie algebras \cites{Kostant61,Sigg96}. Moreover, $\mathcal C_B^\bullet(\mathfrak g)$ can be thought of as the sector consisting of states in which only fermions from the first species are excited and $\frac{1}{2}\Delta_B$ as the operator whose eigenvalues measure their energy. It is then natural to look at \eqref{eq:21} as the partition function for this sector. Furthermore, since $\ker(\Delta_B)\cong H^\bullet_B(\mathfrak g)$, we obtain that for any metric pair $(\mathfrak g,g)$ that $Z_{\mathfrak g,g}(0,t)$ is equal to the independence polynomial $I(\mathfrak g,t)$.
\end{rem}

\begin{example}
Let $G$ be a finite simple graph. It follows from \eqref{eq:18} that $\Delta_B$ is diagonalized by the monomial basis of $\mathcal C^\bullet_B(\mathfrak g)$ and that the eigenvalue corresponding to the monomial $\omega=y_{i_1}^*\cdots y_{i_k}^*$ is equal to the number of edges in the induced subgraph $G[\{i_1,\ldots,i_k\}]$. Since $\deg(\omega)=k\omega$, we conclude that 
\begin{equation}
Z_{\mathcal L(G),g_G}(s,t)=F(G,1,1,s,t)\,,
\end{equation}
where $F$ denotes the generalized subgraph counting polynomial \eqref{eq:4}.
\end{example}

\begin{mydef}
Let $[n]=\{1,\ldots,n\}$. The {\it Johnson graph} $J(n,k)$ is the graph whose vertices are the $k$-element subsets of $[n]$, i.e. $V(J(n,k))=\binom{[n]}{k}$, and such that $S_1,S_2\subseteq [n]$ share an edge if and only if $|S_1\cap S_2|=k-1$.
\end{mydef}

\begin{lem}[\cite{BrouwerCohenNeumaier89}]\label{lem:38}
If $A_{n,k}$ is the adjacency matrix of the Johnson Graph $J(n,k)$, then
\begin{equation}
\det(xI-A_{n,k})=\prod_{j=0}^{\min(k,n-k)}(x-\theta_{n,k,j})^{f_{n,j}}\,,
\end{equation}
where $\theta_{n,k,j}=(k-j)(n-k-j)-j$, $f_{n,j}=\binom{n}{j}-\binom{n}{j-1}$, and $I$ stands for the $\binom{n}{k}\times \binom{n}{k}$ identity matrix.
\end{lem}

\begin{theorem}\label{thm:40}
Let $(\mathfrak h_n,h)$ be the metric pair such that the basis of Example \ref{ex:5} is orthonormal with respect to $h$. Then
\begin{equation}
Z_{\mathfrak h_n,g}(s,t)=\sum_{k=0}^n\sum_{m=0}^{n-k}\sum_{j=0}^{\min(m,n-k-m)} N_{n,k,j}\quad s^{(m-j)(n-k-m-j+1)}t^{2m+k}\,,
\end{equation}
where $N_{n,k,0}=2^k\binom{n}{k}$ and
\begin{equation}
N_{n,k,j}=2^k\binom{n}{k\,,j\,,n-k-j}\frac{n-k-2j+1}{n-k-j+1}
\end{equation}
whenever $j\ge 1$.
\end{theorem}

\begin{proof}
For every $i\in\{1,\ldots,n\}$ let $\omega_i=y_i^*y_{i+n}^*$ and $D_i=\frac{\partial}{\partial y_{i+n}^*}\frac{\partial}{\partial y_i^*}$. In this notation, \eqref{eq:17} specializes to $\Delta_B=\sum_{i=1}^n \omega_i D_i$. A monomial in $\mathcal C^\bullet(\mathfrak h_n)$ is in $\ker(\Delta_B)$ if and only if it is not divisible by any of the $\omega_i$. In particular, there are $\binom{n}{k}2^k$ such monomials of degree $k$. Let $\mathcal S=\mathbb R[\omega_1,\ldots,\omega_n]$ be the space of (commutative) polynomials in the quadratic monomials $\omega_i$. If $\mathcal S(m)$ is the subspace of $\mathcal S$ consisting of degree $m$ polynomials, then $\mathcal C^\bullet_B(\mathfrak h_n)$ decomposes as the direct sum of subspaces of the form $\mathcal S(m)\rho$ labeled by $m\in\{0,\ldots,n\}$ and by monomials $\rho\in \ker(\Delta_B)$. By inspection, $\Delta_B$ preserves each $\mathcal S(m)\rho$ and its matrix representative with respect to any monomials basis is $mI+A_{n-\deg(\rho),m}$, where $I$ is the identity matrix of size $\binom{n-\deg(\rho)}{m}$ and $A_{n-\deg(\rho),m}$ is the adjacency matrix of the Johnson graph $J(n-\deg(\rho),m)$. The result then follows from Lemma \ref{lem:38} and straightforward calculations with binomial coefficients.

\begin{example}
In the notation of Example \ref{ex:35}, it follows from Theorem \ref{thm:40} that
\begin{equation}
    Z_{\mathfrak h_2,h}(s,t)=1 + 4t + 5t^2 + s^2 t^2 + 4st^3 + s^2 t^4\,.
\end{equation}
On the other hand,
\begin{equation}
Z_{\mathfrak h_2,h'}(s,t)=Z_{\mathfrak g,g}(s,t)=1 + 4t + 5t^2 + s^3 t^2 + 2s^{(3-\sqrt{5})/2}t^3 + 2s^{(3+\sqrt{5})/2}t^3+s^3t^4\,,
\end{equation}
showing at once that the basic partition function is not necessarily a polynomial and that, in general, different choices of inner product lead to different basic partition functions. 
\end{example}

\end{proof}

\subsection*{Acknowledgements} The work on this paper was supported by the National Science Foundation grant number DMS1950015. We are very grateful to Glenn Hurlbert and Craig Larson for illuminating correspondence. High performance computing resources provided by the High Performance Research Computing (HPRC) Core Facility at
Virginia Commonwealth University were used for conducting the research reported in this work.

\vskip.1in\noindent
\address{Marco Aldi\\
Department of Mathematics and Applied Mathematics\\
Virginia Commonwealth University\\
Richmond, VA 23284, USA\\
\email{maldi2@vcu.edu}}

\vskip.1in\noindent
\address{Thor Gabrielsen\\
Department of Mathematics\\
Colby College\\
Waterville, ME 04901\\
\email{trgabr26@colby.edu}}

\vskip.1in\noindent
\address{Daniele Grandini\\
Department of Mathematics and Economics\\
Virginia State University\\
Petersburg, VA 23806, USA\\
\email{dgrandini@vsu.edu}}

\vskip.1in\noindent
\address{Joy Harris\\
Department of Mathematics\\
University of Georgia\\
Athens, GA 30602, USA\\
\email{jharris80412@gmail.com}}

\vskip.1in\noindent
\address{Kyle Kelley\\
Department of Mathematics and Statistics\\
Kenyon College\\
Gambier, OH 43022\\
\email{KyleAKelley@pm.me}}

\end{document}